\documentclass[12pt]{article}
\usepackage{enumitem}

\usepackage{graphicx}
\usepackage{amsfonts}
\usepackage{amsmath, amsthm}
\usepackage{amssymb}
\usepackage[ansinew]{inputenc}
\usepackage{latexsym}
\usepackage{tabularx}
\usepackage{makeidx}
\usepackage{color}

\usepackage{caption}

\allowdisplaybreaks
\newtheorem{theorem}{Theorem}[section]
\newtheorem{lemma}[theorem]{Lemma}

\theoremstyle{definition}
\newtheorem{defi}{Definition}[section]

\theoremstyle{remark}
\newtheorem{remark}{Remark}

\newcommand{\NN}{{\mathbb N}}

\newcommand{\RR}{{\mathbb R}}

\newcommand{\Aa}{{\mathcal A}}

\begin{document}

\title{\scshape{On the star discrepancy of sequences in the unit interval}}
\author{Gerhard Larcher\thanks{is partially supported by the Austrian Science Fund (FWF), Project P21943 and Project F5507-N26, which is a part of the Special Research Program ``Quasi-Monte Carlo Methods: Theory and Applications"}     }
\date{}
\maketitle

\begin{abstract}
It is known that there is a constant $c > 0$ such that for every sequence $x_1, x_2, \ldots$ in $[0,1)$ we have for the star discrepancy $D_N^*$ of the first $N$ elements of the sequence that $N  D_N^* \ge c \cdot \log N$ holds for infinitely many $N$. Let $c^*$ be the supremum of all such $c$ with this property. We show $c^* > 0.0646363$, thereby improving the until now known estimates.
\end{abstract}

\section{Introduction and statement of the result}\label{sec1}

Let $x_1, x_2, \ldots$ be a point sequence in $[0,1)$. By $D_N^*$ we denote the star discrepancy of the first $N$ elements of the sequence, i.e., 
$$ D_N^* = \sup_{x \in [0,1]} \left| \frac{\Aa_N (x)}{N} - x \right| , \quad \textnormal{where} $$
$ \Aa_N (x) := \# \{ 1 \le n \le N \, \, \vline \, \, x_n < x \}$.\\
The sequence $x_1, x_2, \ldots$ is uniformly distributed in $[0,1)$ iff $\lim_{N \rightarrow \infty} D_N^* = 0$.

In 1972 W.M.Schmidt \cite{Schmidt1972} has shown that there is a positive constant $c$ such that for all sequences $x_1, x_2, \ldots$ in $[0,1)$ we have
$$ D_N^* > c \cdot \frac{\log N}{N}  $$
for infinitely many $N$.\\
The order $\frac{\log N}{N}$ in this result is best possible. There are many sequences known for which $D_N^* \le c' \cdot \frac{\log N}{N}$ for a certain constant $c'$ and for all $N$ holds.

So it makes sense to define the ``one-dimensional star discrepancy constant'' $c^*$ to be the supremum over all $c$ such that 
$$ D_N^* > c \cdot \frac{\log N}{N} $$
holds for all sequences $x_1, x_2, \ldots$ in $[0,1)$ for infinitely many $N$. Or, in other words
$$ c^* := \inf_{w} \limsup_{N \rightarrow \infty} \frac{N  D_N^* (w)}{\log N} \quad , $$
where the infimum is taken over all sequences $w = x_1, x_2, \ldots$ in $[0,1)$, and $D_N^* (w)$ denotes the star discrepancy of the first $N$ elements of $w$. 

The currently best known estimates for $c^*$ are
$$ 0.06015\ldots \le c^* \le 0.222\ldots  . $$

The upper bound was given by Ostromoukhov \cite{Ostromoukhov2009} (thereby slightly improving earlier results of Faure (see for example \cite{Faure1992})). The lower bound was given by B\'ejian \cite{Bejian1982} . (In fact B\'ejian derives his bound for $c^*$ from a bound for the corresponding constant with respect to extreme discrepancy.)

It is the aim of this paper to give a simple, more illustrative proof of the result of B\'ejian on $c^*$ with an even sharper lower bound for $c^*$. 

We will prove

\begin{theorem}\label{theo1}
$$ c^* \ge 0.0646363\ldots $$
\end{theorem}

In Section \ref{sec2} we will give some auxiliary results. The proof of Theorem \ref{theo1} then follows in Section \ref{sec3}. 
The idea of the proof follows a method introduced by Liardet \cite{Liardet1979} which was also used by Tijdeman and Wagner in \cite{TiWag1980}.

\section{Auxiliary results}\label{sec2}

\begin{lemma}\label{lem1}
For any set $A$, any subsets $A_0, A_2$ of $A$ and any function $f: A \rightarrow \RR$ we have

\begin{align*}
& \max_{n \in A} \, f (n) - \min_{n \in A} \, f (n) \ge \\
& \frac{1}{2} \left( \max_{n \in A_2} \, f (n) - \min_{n \in A_2} \, f (n) \right) + \frac{1}{2} \left( \max_{n \in A_0} \, f (n) - \min_{n \in A_0} \, f (n) \right) + \\
+ & \frac{1}{2} \left| \max_{n \in A_2} \, f(n) - \max_{n \in A_0} \, f (n) \right|  +  \frac{1}{2} \left| \min_{n \in A_2} \, f(n) - \min_{n \in A_0} \, f (n) \right| \quad .
\end{align*}
\end{lemma}

\begin{proof}

This is quite elementary.

\end{proof} 

Consider now a finite point set $x_1, x_2, \ldots x_N$ in $[0,1)$ with $N = [a^t]$  
, for some real $a$ with $3 < a < 4$ and some $t \in \NN$.
Let $A$ be the index-set $A = \{ 1, 2, \ldots, N \}$, and $A_0, A_1, A_2$ be the index-subsets \newline $A_0 = \{1, 2, \ldots a^{t-1}\}, A_2 = \{a^t - a^{t-1} + 1, a^t - a^{t-1} + 2, \ldots, a^t \}$ and $A_1 = A \backslash (A_0 \cup A_2)$.

For $x \in [0,1)$ we consider the discrepancy function
$$ D_n (x):= \# \{ i \le n \, \vline \, x_i < x \} - n \cdot x = \Aa_n (x) - n \cdot x .$$

In the proof of Theorem \ref{theo1} we will have to deal with the function
$$ f(x) := \max_{n \in A_2} D_n (x) - \max_{n \in A_0} D_n (x) \quad .$$

We start with discussing some basic properties of $f(x)$.\\

We have
$$ f (x) = \Aa_{n_2} (x) - \Aa_{n_0} (x) - (n_2 - n_0) \cdot x $$
for some $n_i = n_i (x) \in A_i \, ; \quad i = 0, 2$. 

\noindent Note that
$$ (a^t - a^{t-1}) - a^{t-1} \le n_2 - n_0 \le a^t $$
always. $\left| f (x) \right|$ is bounded by $a^t$.

The function $f$ is for all $x \neq x_j \, \, (j = 1, \ldots, N)$ continuous (note that $n_i (x)$ can change their values also at $x \neq x_j$, but $f$ stays continuous for these $x$). 

\noindent Hence for $x \neq x_j \, \, (j= 1, \ldots, N) \, f$ is piecewise linear and continuous with negative slope between $-a^{t-1} (a-2)$ and $-a^t$. Consequently $f$ has at most $a^t$ discontinuities, namely at most for $x = x_j$ for some $j = 1, \ldots, N = a^t$. In $x = x_j$ we have $\lim_{y \rightarrow x^-} f (y) = f (x)$. Consider now $x = x_j$ for some $j$ with $j \in A_1$.

\noindent Then in $x_j$ the value $\Aa_{n_0} (x)$ does not change so $\Aa_{n_0} (x) - n_0 x$ has no jump in $x$, whereas $\Aa_{n} (x)$ increases by one for {\bf all} $n \in A_2$, hence $\Aa_{n_2} (x) - n_2 x$ and therefore $f (x)$ has a jump of height $1$ in $x_j$.

\noindent Hence $f(x)$ has at least $a^t - 2 a^{t-1}$ jumps of height at least $1$.

\begin{defi}
Let $a \in \RR, a > 2, t \in \NN$. Let $f: [0,1) \rightarrow \RR$ be a function with the following properties
\begin{itemize}
\item[i)] $f(0) = f (1) = 0$
\item[ii)] $f$ is piecewise monotonically decreasing and piecewise linear and its absolute value is bounded by $a^t$.
\item[iii)] $f$ has at most $a^t$ discontinuities. In $a$ discontinuity $x$ there is always a positive jump and $f (x) = \lim_{y \rightarrow x^-} f (y)$. 
\item[iv)] $f$ has at least $a^{t-1} (a-2)$ discontinuities in which $f$ has a jump of at least $1$
\item[v)] the slope of $f$ is always between $-a^t$ and $- (a-2) a^{t-1}$
\end{itemize}
Then we say: $f$ is admissible.


\end{defi}

\begin{lemma}\label{lem2}
There exists an $f^*: [0,1] \rightarrow \RR$ admissible such that 
$$ \int_0^1 \left| f^* (t) \right| dt = \min_{f \textnormal{admissible}} \int_0^1 \left| f (t) \right| dt .$$
\end{lemma}

\begin{proof}

This follows immediatly from the obious fact that the set of admissible functions is closed with respect to pointwise convergence.

\end{proof}

\begin{lemma}\label{lem3}
Let $f^*$ as defined in Lemma \ref{lem2}.

Let $f^*$ have two successive discontinuities in $a_1$ and $a_2, 0 < a_1 < a_2 < 1$. Then $f^*$ has a zero in the interval $(a_1, a_2)$.
\end{lemma}\newpage

\begin{proof}
Assume in the contrary that $f^* (x) > 0$ for all $x \in (a_1, a_2)$ (see Figure \ref{figure1}).

  
\begin{figure}[h!]
  \begin{center}
       \includegraphics[scale=0.175]{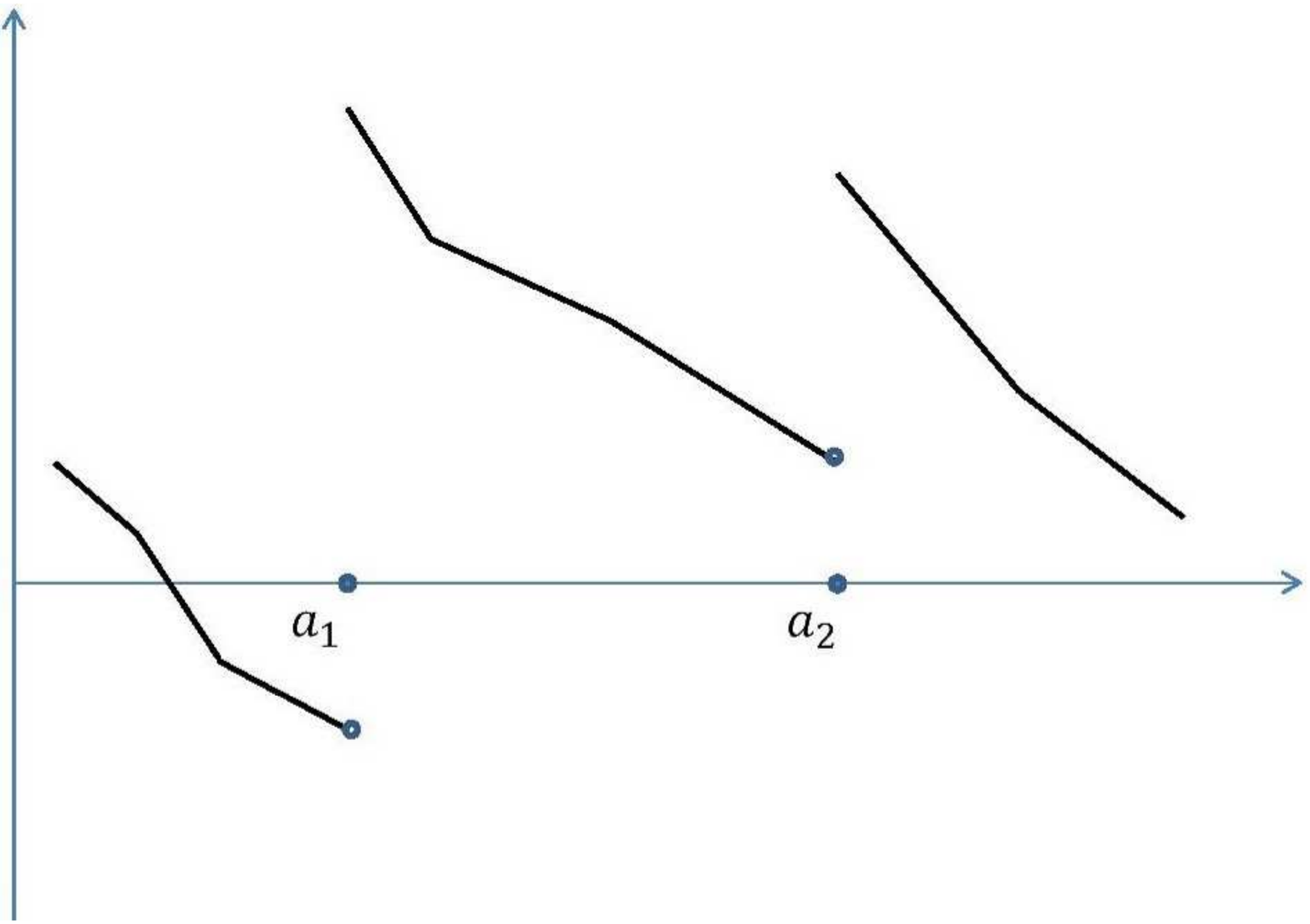}
  \end{center}
\caption{} \label{figure1}
  \end{figure}

\noindent If we replace for a $\delta > 0$ small enough $f^* (x)$ for $y \in (a_2, a_2+\delta]$ by $f^* (y) + f^* (a_2) - \lim_{x \rightarrow a_2^+} f^* (x)$ (see Figure \ref{figure2}), then the resulting function $\tilde{f}$ still is admissible and 
$$ \int_0^1 \left| \tilde{f} (x) \right| dx < \int_0^1 \left| f^* (x) \right| dx$$
which is a contradiction. (The argument also works if $f^* (a_2) = 0$ and $\delta$ is small enough. If $f^* < 0$ in ($a_1, a_2$) we use an analogous argument.)\vspace{1em}

%

\begin{figure}[h!]
  \begin{center}
       \includegraphics[scale=0.17]{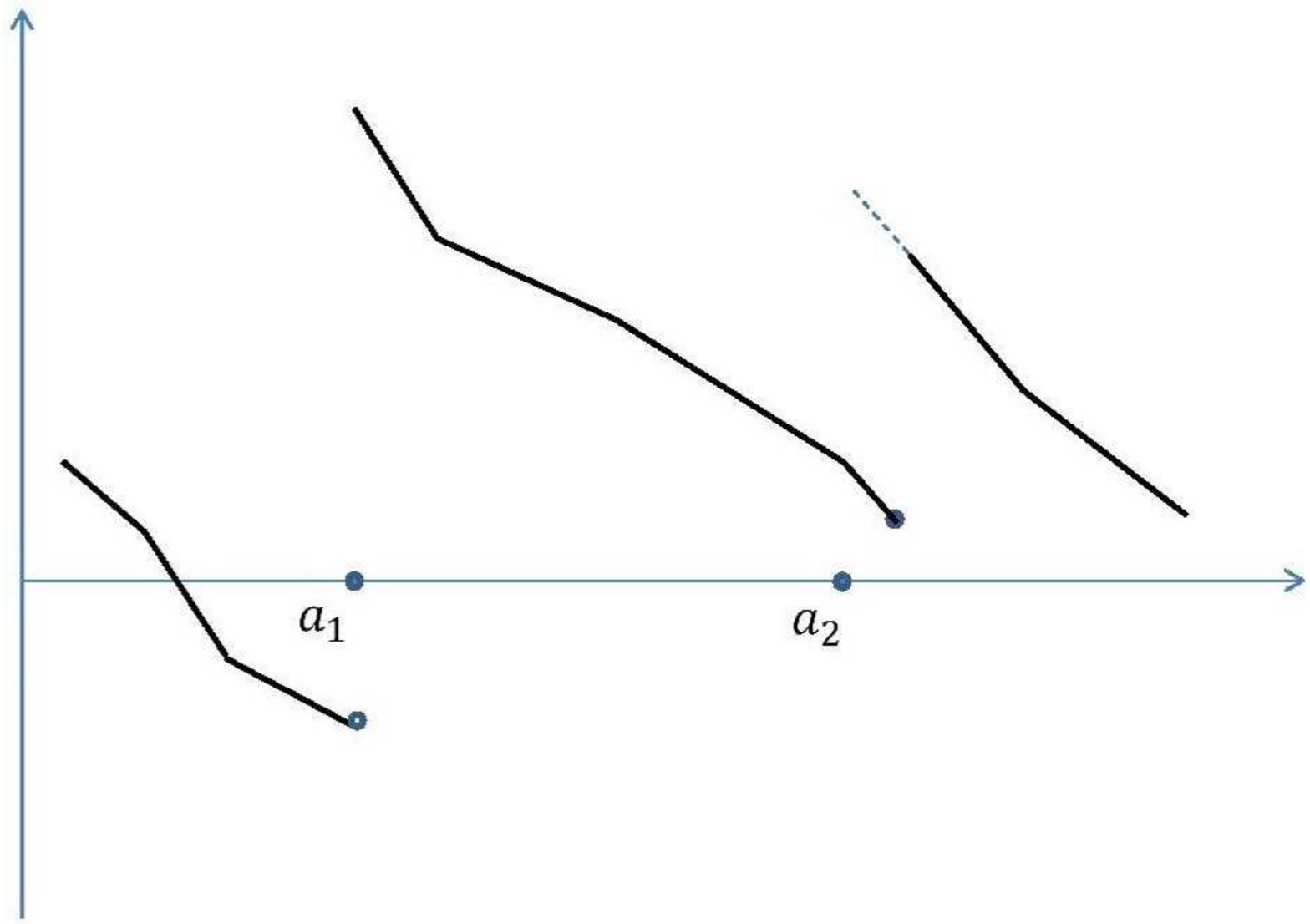}
  \end{center}
\caption{} \label{figure2}
  \end{figure}
\end{proof}

\noindent By Lemma \ref{lem3} it follows that $f^*$ consists of parts $Q$ of the form like in Figure \ref{figure3}

%

\begin{figure}[h!]
  \begin{center}
       \includegraphics[scale=0.175]{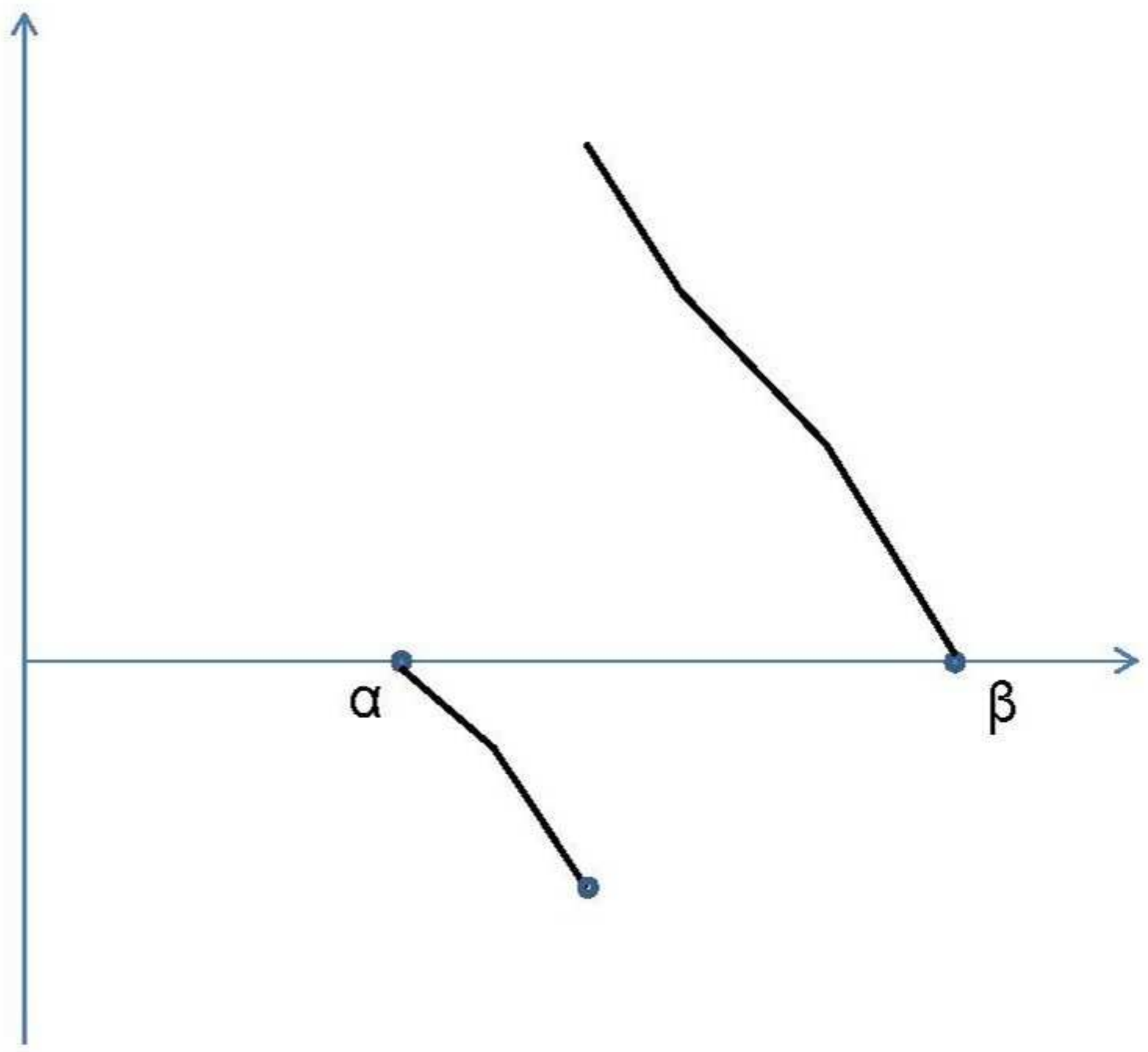}
  \end{center}
\caption{}\label{figure3}
  \end{figure}

\noindent i.e., $f^* (\alpha) = f^* (\beta) = 0$ and $f^*$ has exactly one discontinuity in $(\alpha, \beta)$.

 Next we show

\begin{lemma}\label{lem4}
Let $f^*$ as defined in Lemma \ref{lem2}. Then $f^*$ has exactly $a^t$ discontinuities.
\end{lemma}

\begin{proof}

Assume in the contrary that $f^*$ has less than $a^t$ discontinuities, then we can construct an admissible $\tilde{f}$ with 
$$ \int_0^1  \left| \tilde{f} (x) \right| dx  < \int_0^1 \left| f^* (x) \right| dx $$
by one of the following actions:

if $f$ does not have exactly $a^{t-1}$ discontinuities with jump of height exactly equal to $1$ in each discontinuity, then there exists a part $Q$ (like in Figure \ref{figure3}) with a jump with height less than one, or with height larger than one, or $f$ has more than $a^{t-1}$ discontinuities with a jump with height exactly equal to one. In the first two cases consider this $Q$, in the third case consider an arbitrary $Q$ of the form  like in Figure \ref{figure3} and replace $f^*$ in $Q$ by any $\tilde{f}$ as illustrated in Figure \ref{figure4}.

  
\begin{figure}[h!]
  \begin{center}
       \includegraphics[scale=0.2]{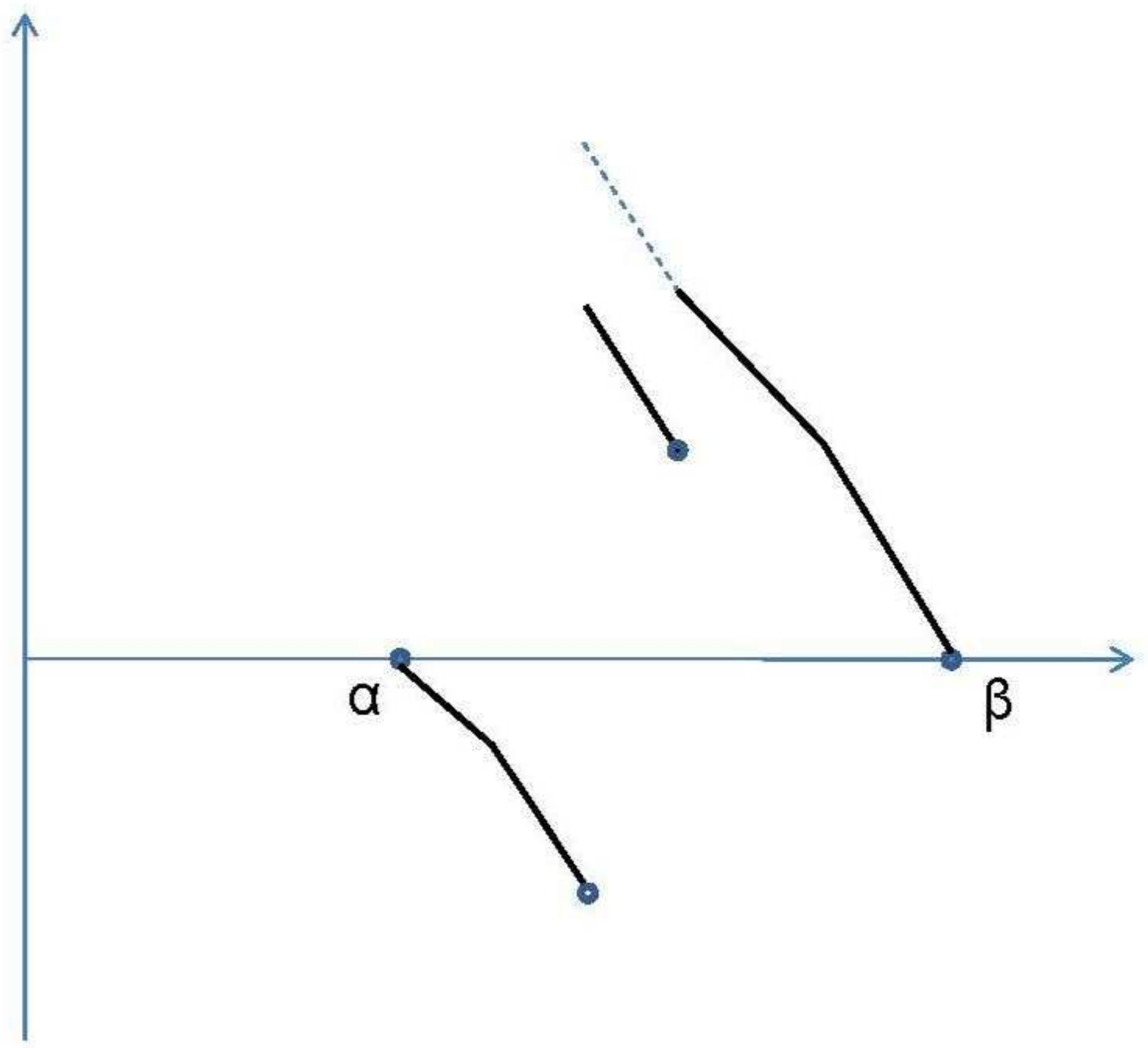}
  \end{center}
\caption{}\label{figure4}
  \end{figure}

\noindent (In the second of the above cases we have to take care that the height of the reduced left jump still is at least one.)

\noindent If $f^*$ has exactly $a^{t-1}$ discontinuities with jumps exactly equal to one, then $f^*$ cannot everywhere have slope equal to $-a^t$ as an easy calculation shows. So there exists an interval $[\gamma, \delta] \subseteq [0,1)$ such that $f^*$ on $[\gamma, \delta]$ has slope larger than $-a^t$ and such that either $f^* (x) > 0$ for all $x \in [\gamma, \delta]$ or $f^* (x) < 0$ for all $x \in [\gamma, \delta]$.\\ 
In the first case we replace $f^*$ on $[\gamma, \delta']$ by $\tilde{f}(x):= f^* (\gamma) + (x-\gamma) \cdot (-a^t)$, where $\delta'$ with $\gamma < \delta' \le \delta$ is maximal such that $\tilde{f}(x) \ge 0$ for $x \in [\gamma, \delta']$.\\
In the second case we replace $f^*$ on $[\gamma', \delta]$ by $\tilde{f} (x):= f^* (\delta) + (x-\delta) \cdot (-a^t)$, where $\gamma'$ with $\gamma \le \gamma' < \delta$ is minimal such that $\tilde{f}(x) \le 0$ for $x \in [\gamma', \delta]$. \\
In all cases $\tilde{f}$ is admissible and obviously $\int_0^1 \left| \tilde{f}(x) \right| dx < \int_0^1 \left| f^*(x) \right| dx $, a contradiction. 
\end{proof}

So $f^*$ (as defined in Lemma \ref{lem2}) consists of $(a-2) a^{t-1}$ parts $Q'$ with a jump of height at least $1$, and of $2  a^{t-1}$ parts $Q''$ with jumps of arbitrary height.

\begin{lemma}\label{lem5}
Let $f^*$ as defined in Lemma \ref{lem2}. Then a part $Q''$ of $f^*$ with a jump of arbitrary height, and defined on an interval $[\alpha, \beta]$ must be of the form $f^* (\alpha) = f^* (\beta) = 0$, $f^*$ has a jump in $\frac{\alpha + \beta}{2}$ and the slope of $f^*$ on $Q''$ is $-(a-2) a^{t-1}$ everywhere in $[\alpha, \beta]$.
\end{lemma}

\begin{proof}
This is obvious. Indeed, assume that $f^*$ were of arbitrary other slope with a jump in $\gamma \, (\alpha < \gamma < \beta)$, then  
\begin{align*}
\tilde{f}(x):= \left\{
\begin{array}{ll}
- (a-2) a^{t-1} (x-\alpha) & \, \textnormal{for} \quad \alpha \le x \le \gamma \\
- (a-2) a^{t-1} (x-\beta) & \, \textnormal{for} \quad \gamma < x \le \beta \\
f^* (x) & \, \textnormal{else}
\end{array}
\right.
\end{align*}

\noindent satisfies $\int_0^1 \left| \tilde{f}(x) \right| dx < \int_0^1 \left| f^*(x) \right| dx $ (see Figure \ref{figure5}), and $\int_0^1 \left| \tilde{f}(x) \right| dx$ becomes minimal if $\gamma = \frac{\alpha + \beta}{2}$ as an easy calculation shows.

%

\begin{figure}[h!]
  \begin{center}
       \includegraphics[scale=0.17]{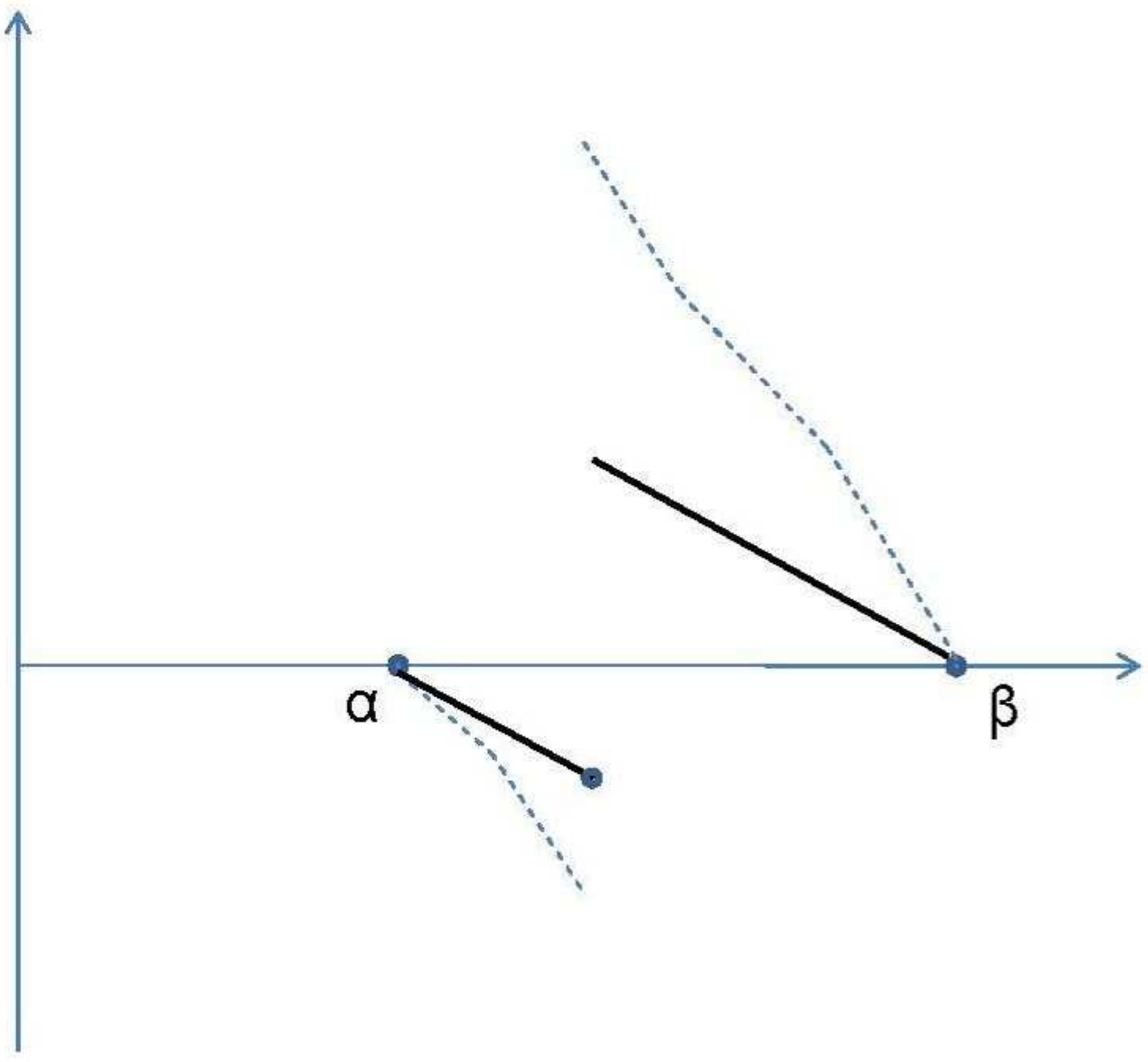}
  \end{center}
\caption{}\label{figure5}
  \end{figure}
\end{proof} 

Next we consider parts $Q'$ of $f^*$ on an interval $[\alpha, \beta]$ with a jump of height at least $1$ in this interval and determine the form of $f^*$ on such $Q'$.

\begin{lemma}\label{lem6}
Let $f^*$  be defined like in Lemma \ref{lem2} and let $Q'$ be like defined after the proof of Lemma \ref{lem4}. Assume that the place $\gamma \in [\alpha, \beta]$ of the jump, and $- \delta := f^* (\gamma)$ and $\tau := \lim_{x \rightarrow \gamma^+} f^* (x)$ are given.\\
Note that necessarily $(\gamma - \alpha) (a-2) a^{t-1} \le \delta \le (\gamma - \alpha) a^t$ and $(\beta - \gamma) (a-2)  a^{t-1} \le \tau \le (\beta - \gamma) a^t$.\\
Then there are uniquely determined points $x_1 \in [\alpha, \gamma]$ and $x_2 \in [\gamma, \beta]$ such that the following (admissible) function $\tilde{f}$ is well-defined:
$$ \tilde{f}(\alpha) = \tilde{f} (\beta) = 0, \quad \tilde{f}(\gamma) = - \delta, \quad \lim_{x \rightarrow \gamma^+} \tilde{f}(x) = \tau ,  $$
$\tilde{f}(x)$ has slope $- (a-2) a^{t-1}$ in $[\alpha, x_1] \cup [x_2, \beta], \, \tilde{f}(x)$ has slope $-a^t$ in $[x_1, x_2]$ (see Figure \ref{figure6}).


\begin{figure}[h!]
  \begin{center}
       \includegraphics[scale=0.185]{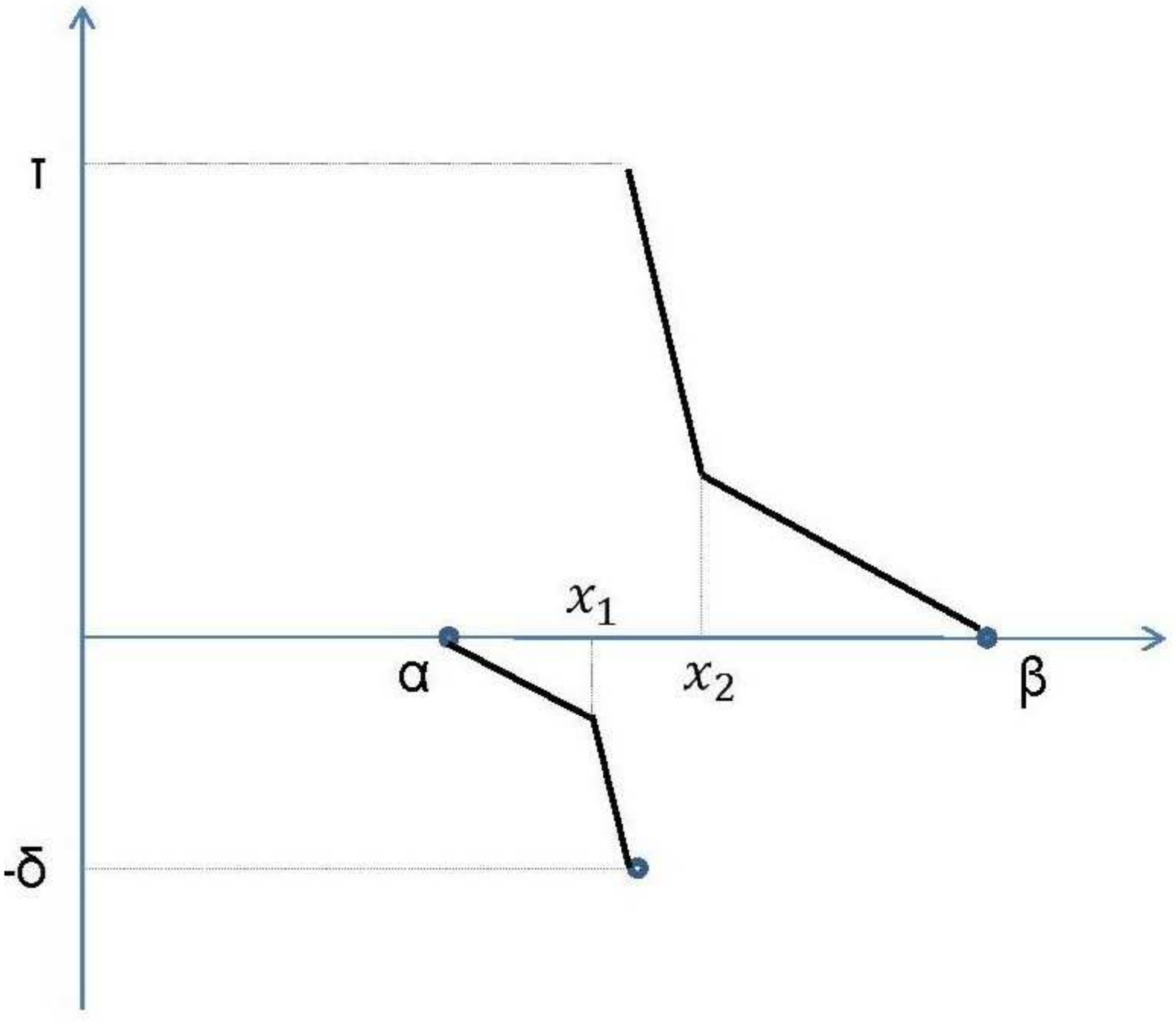}
  \end{center}
\caption{}\label{figure6}
  \end{figure}\newpage

\noindent Indeed $x_1 = \frac{\gamma a^t - \alpha (a-2) a^{t-1} - \delta}{2 a^{t-1}}$ and $x_2 = \frac{\tau - \beta (a-2) a^{t-1} + \gamma a^\tau}{2 a^{t-1}}$ . Then $f^*$ has to equal $\tilde{f}$ on $[\alpha, \beta]$. 
\end{lemma}

\begin{proof}
This is obvious since every admissible $f$ with $f(\alpha) = f (\beta) = 0$, a single jump in $[\alpha, \beta]$ at place $\gamma$, with $f(\gamma) = - \delta$ and $\lim_{x \rightarrow \gamma^+} f(x) = \tau  $ necessarily satisfies 
\begin{align*}
f(x) \le \tilde{f}(x) & \quad \textnormal{for} \quad x \in [\alpha, \gamma] \quad \textnormal{and} \\
f(x) \ge \tilde{f}(x) & \quad \textnormal{for} \quad x \in (\gamma, \beta] .
\end{align*}
\end{proof}

\begin{lemma}\label{lem7}
Let $f^*$ be defined like in Lemma \ref{lem2} and $Q'$ be like defined after the proof of Lemma \ref{lem4}. Then $f^*$ has the form as described in Lemma \ref{lem6} with $\delta + \tau = 1$, i.e., the height of the jump is equal to $1$.
\end{lemma}

\begin{proof}
This is immediately clear for example from Figure \ref{figure7}.

%

\begin{figure}[h!]
  \begin{center}
       \includegraphics[scale=0.18]{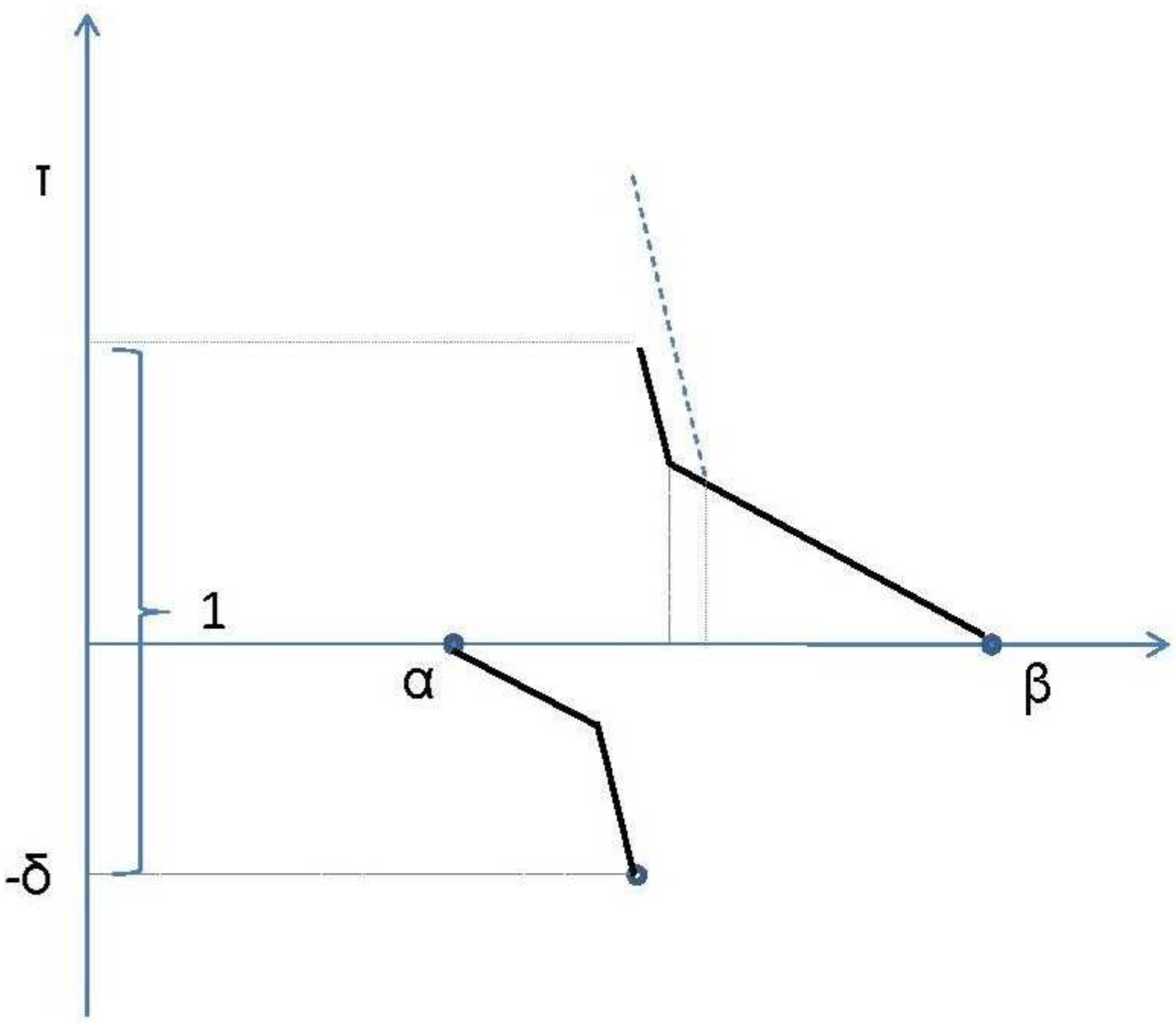}
  \end{center}
\caption{}\label{figure7}
  \end{figure}
\end{proof}

\begin{lemma}\label{lem8}
Let $f^*$ be defined like in Lemma \ref{lem2}, $Q'$ be like defined after the proof of Lemma \ref{lem4}, and $\delta$ like described in Lemma \ref{lem6}. Then $\int_{Q'} \left| f^* (x) \right| dx$ is minimal for 
$$ \delta = \frac{1}{2} + a^{t-1} (\alpha + \beta - 2 \gamma).$$
\end{lemma}

\begin{proof}

This easily follows from minimising the function $\int_{Q'} \left| f^*(x) \right| dx $ which is a quadratic function in $\delta$ with respect to $\delta$.

\end{proof}

\begin{lemma}\label{lem9}
Let $f^*$ be defined like in Lemma \ref{lem2}, $Q'$ be like defined after the proof of  Lemma \ref{lem4}, and $\delta$ like determined by Lemma \ref{lem8}.\\
Then $\int_{Q'} \left| f^*(x) \right| dx $ is minimal for $\gamma = \frac{\alpha + \beta}{2}$ (and hence $\delta = \frac{1}{2}$).
\end{lemma}

\begin{proof}

This again easily follows from minimising the function $\int_{Q'} \left| f^*(x) \right| dx $ which is quadratic in $\gamma$.

\end{proof}

So we know now that $f^*$ consists of $(a-2) a^{t-1}$ parts $Q'$ of the form like in Figure \ref{figure8}

%
%
\begin{figure}[h!]
  \begin{center}
       \includegraphics[scale=0.22]{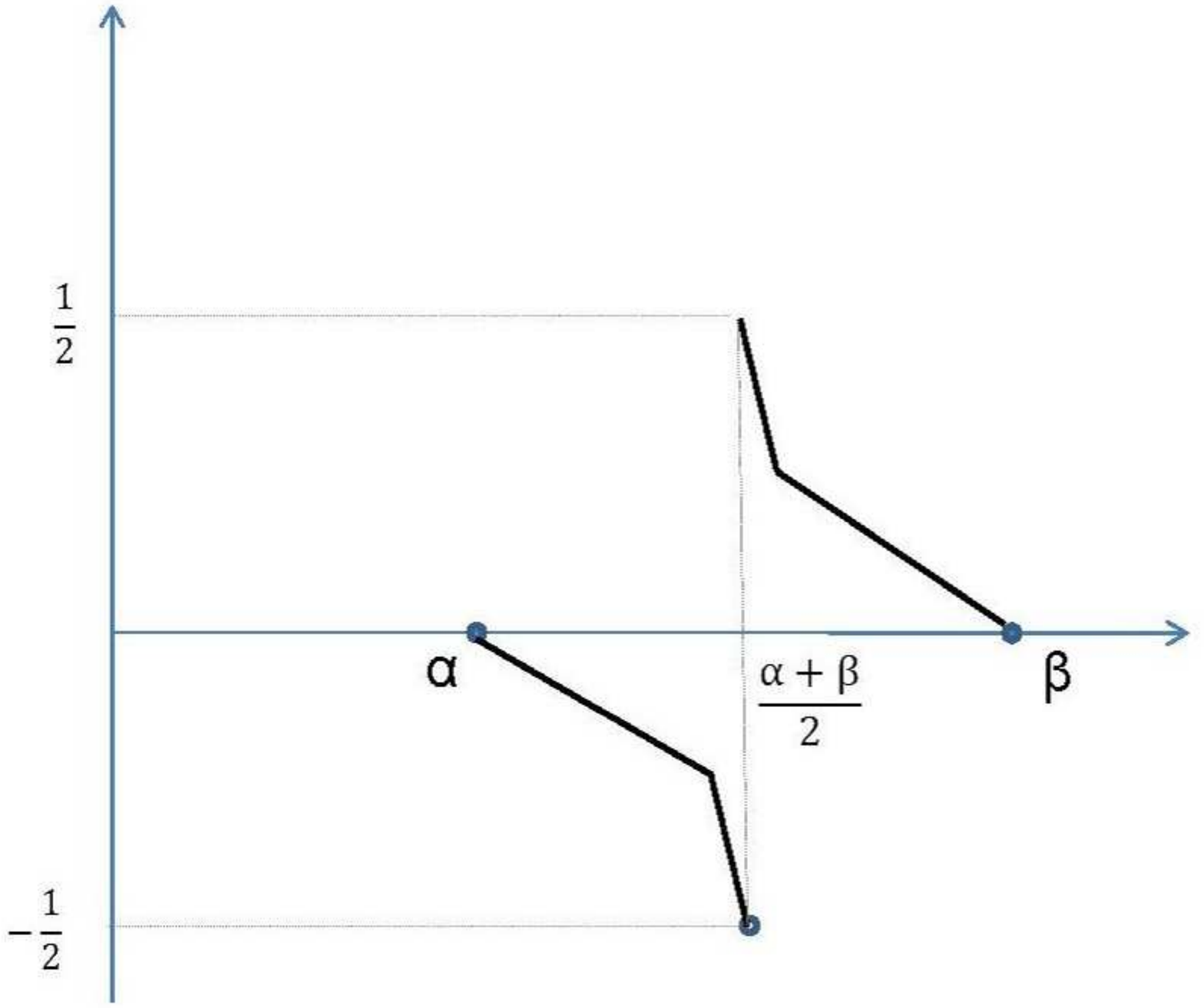}
  \end{center}
\caption{}\label{figure8}
  \end{figure}
   
 
and of $2 a^{t-1}$ parts $Q''$ of the form like in Figure \ref{figure9}


\begin{figure}[h!]
  \begin{center}
       \includegraphics[scale=0.21]{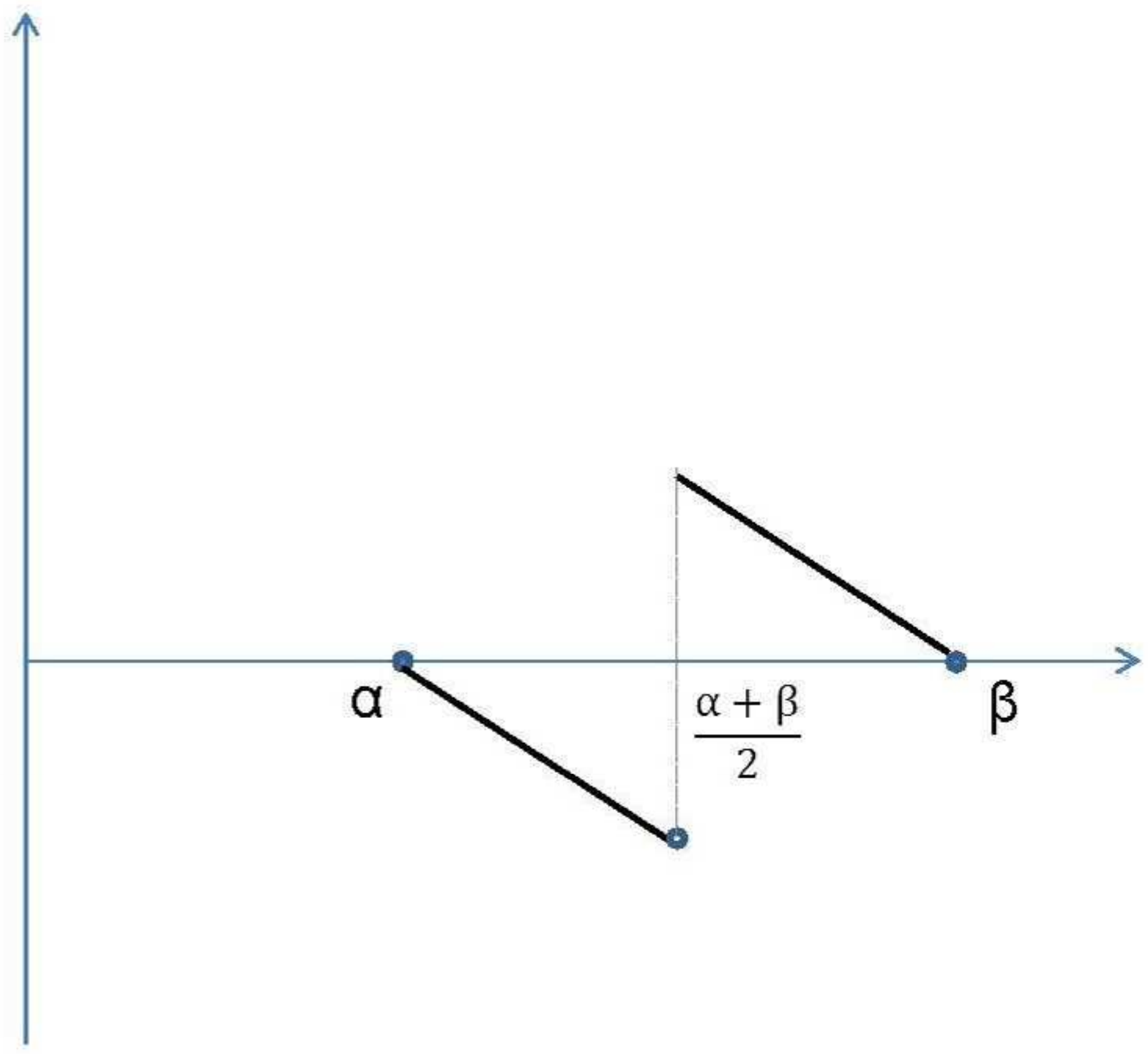}
  \end{center}
\caption{}\label{figure9}
  \end{figure}

\noindent where all linear parts either have minimal slope or maximal slope.
We have (with $\chi := \beta - \alpha$)
\begin{equation}\label{equa1}
\int_{Q'}  \left| f^*(x) \right| dx = \frac{1}{8} \left( a^{1-t} + 4 \chi - 2 a \chi - 2 a^t \chi^2 + a^{1+t} \chi^2 \right)
\end{equation}

\begin{equation}\label{equa2}
\int_{Q''}  \left| f^*(x) \right| dx = \frac{\chi^2}{4} (a-2) a^{t-1}
\end{equation}

\begin{lemma}\label{lem10}
Let $f^*$ be defined like in Lemma \ref{lem2} and $Q'$ be like defined after the proof of  Lemma \ref{lem4}. Let $Q_1', Q_2'$ be two parts of $f^*$ of form $Q'$ with interval lengths $\chi_1, \chi_2$. Let $\bar{\chi}:= \chi_1 + \chi_2$ be given, then
$$ \int_{Q_1' \cup Q_2'}  \left| f^*(x) \right| dx \quad \textnormal{is minimal if } \quad \chi_1 = \chi_2 = \frac{\bar{\chi}}{2} \quad.$$
The same assertion holds for the parts of $f^*$ of form $Q''$ like defined in Lemma \ref{lem5}.
\end{lemma}

\begin{proof}

This again follows by a simple minimisation of a quadratic polynomial.

\end{proof}

The use of admissible functions and the above properties of $f^*$ would suffice with the technique from Section \ref{sec3} to reprove the bound of B\'ejian. To improve his bound we have to introduce the concept of strong admissiblity.

\begin{remark}\label{rem1}
Let again $f(x) := \max_{n \in A_2} D_n (x) - \max_{n \in A_0} D_n (x)$. We consider $f$ on an interval $[\alpha, \beta]$ and assume that $f$ has exactly one jump in $[\alpha, \beta]$, say in $\gamma := x_j \in (\alpha, \beta)$. Further we assume that $x_i \notin [\alpha, \beta]$ for $i \neq j$ with $x_i \neq x_j$. Let $x \in [\alpha, \gamma)$. We again denote $\max_{n \in A_i} D_n (x) =: D_{{n_i}(x)} (x)$ for $i = 0, 2 $ with $x_i \neq x_j$.

Assume that $n_i (x)$ changes its value in $x$, 
\begin{equation}\label{eq:1}
\begin{array}{lllll}
\textnormal{i.e.} \quad D_{{n_i}(x^+)} (x^+) & > \, D_{{n_i}(x^-)} (x^+)\\
\textnormal{and} \quad D_{{n_i}(x^+)} (x) & = \, D_{{n_i}(x^-)} (x)
\end{array}
\end{equation}
Since $x \neq x_i$ for all $i$ we have $\Aa_{{n_i} (x^+)} (x) = \Aa_{{n_i}(x^+)} (x^+)$ and $\Aa_{{n_i}(x^-)}(x) = \Aa_{{n_i}(x^-)} (x^+)$, hence (\ref{eq:1}) is equivalent to 
$$ x^+ \cdot (n_i (x^-) - n_i (x^+)) > x \cdot (n_i (x^-) - n_i (x^+)) $$
and therefore $n_i (x^-) > n_i (x^+)$.\\
So $n_2(x)$ and $n_0 (x)$ are monotonically decreasing in $[\alpha, \gamma)$. 

The slope of $f (x)$ is given by $- (n_2 (x) - n_0 (x))$.\\
We have $n_i (x) \in A_i$, and $A_i$ is an interval of length $a^{t-1}$. Hence the slope of $f(x)$ and $f(x')$ because of the monotonicity of $n_0 (x)$ and $n_2 (x)$ can differ for $x, x' \in [\alpha, \gamma)$ at most by $a^{t-1}$.

So let $- (a^t - v \, a^{t-1})$ be the minimal slope of $f(x)$ for some $x \in [\alpha, \gamma)$, then $v \in [0,2]$, and the maximal slope of $f(x)$ for some $x \in [\alpha, \gamma)$ is at most $$\min \left( - (a^t - (v+1)  a^{t-1}), - (a^t - 2 a^{t-1}) \right)\quad .$$

Of course the same also holds on the interval $(\gamma, \beta]$.\\

We will call this property of $f$ ``condition A".
\end{remark}

\begin{defi}
For given $a$ and $t$, the functions $f: [0, 1] \rightarrow \RR$ which are admissible and which additionally satisfy condition A as described in Remark \ref{rem1}, are called strongly admissible.
\end{defi}

\begin{lemma}\label{lem11}
There exists an $f^{**}: [0, 1] \rightarrow \RR$ strongly admissible such that
$$ \int_0^1 \left| f^{**}(t) \right| dt = \min_{f \, \textnormal{strongly admissible}} \int_0^1 \left| f(t) \right| dt \quad .$$
\end{lemma}

\begin{proof}

This again follows immediatly from the again obvious fact that the set of strongly admissible functions is closed with respect to pointwise convergence. 

\end{proof}

Note that for the strongly admissible $f^{**}$ as defined in Lemma \ref{lem11} we can deduce the same properties as for the admissible $f^*$ defined in Lemma \ref{lem2}, as were given in Lemmas \ref{lem3}, \ref{lem4} and \ref{lem5}.

Because of condition A, the property of Lemma \ref{lem6} cannot hold for $f^{**}$.\\
Instead we have

\begin{lemma}\label{lem12}
Let $f^{**}$ be defined like in Lemma \ref{lem11}. Then let $Q'$ be a part of $f^{**}$ as defined in the proof of Lemma \ref{lem4}. 
Assume that the place $\gamma \in [\alpha, \beta]$ of the jump, and $- \delta := f^{**} (\gamma)$ and $\tau := \lim_{x \rightarrow \gamma^+} f^{**} (x)$ are given.
Let $s_m := - (a^t - v \, a^{t-1})$ be the minimal slope of $f(x)$ on $[\alpha, \gamma)$ and let $s_M := \min \left( - (a^t - (v+1) a^{t-1}), - (a^t - 2a^{t-1})\right)$. The maximal slope of $f(x)$ by condition A) is at most $s_M$.

Note that necessarily
$$ (\gamma - \alpha)  s_m \le - \delta \le (\gamma - \alpha) s_M \quad .$$
Then there is a uniquely determined point $x_1 \in [\alpha, \gamma]$ such that the following strongly admissible function $\tilde{f}_v$ is well-defined:
$\tilde{f}_v (\alpha) = 0, \, \, \tilde{f}_v (\gamma) = - \delta, \, \, \tilde{f}_v (x)$ has slope $s_M$ in $[\alpha, x_1)$ and slope $s_m$ in $[x_1, \gamma)$. (See Figure \ref{figure10}.)


\begin{figure}[h!]
  \begin{center}
       \includegraphics[scale=0.22]{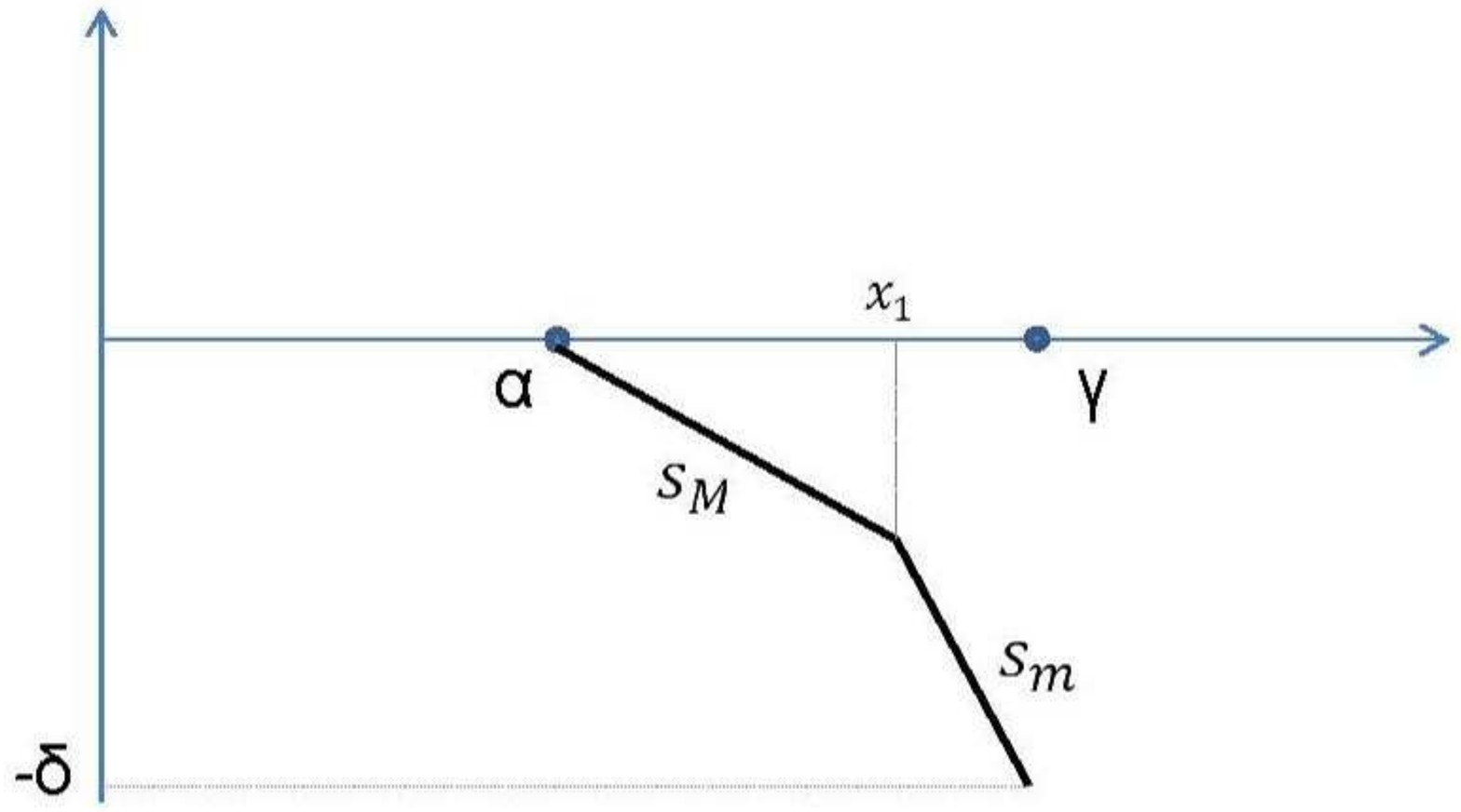}
  \end{center}
\caption{}\label{figure10}
  \end{figure}

Indeed $x_1 = \frac{\delta + \alpha s_M - \gamma s_m}{s_M - s_m}$.
Then $f^{**}$ has to equal on $[\alpha, \gamma]$ the function $\tilde{f}_v$  for some $v \in [0, 1]$.

\end{lemma}

\begin{proof}

This can be deduced as follows:\\
Let first $v \in [0, 2]$ (instead of $[0,1]$) be such that the minimal slope $s_m^*$ of $f^{**}$ in $[\alpha, \gamma]$ is $$s_m^* = - (a^t - v \, a^{t-1}) \quad.$$ Then the maximal slope $s_M^*$ of $f^{**}$ in $[\alpha, \gamma]$ is at most $$s_M^* := \min \left( - (a^t - (v+1) a^{t-1}), - (a^t  - 2 a^{t-1}) \right) \quad.$$ 
Hence obviously $f^{**} (x) \le \tilde{f}_v (x)$ for all $x \in [\alpha, \gamma]$ and therefore $\int_{Q'} \left| f^{**}(x) \right| dx \ge \int_{Q'} \left| \tilde{f}_v (x) \right| dx$.

Hence $f^{**} = \tilde{f}_v$ for some $v \in [0, 2]$.
But 
$$ \int_\alpha^\gamma \left| \tilde{f}_1(x) \right| dx < \int_\alpha^\gamma \left| \tilde{f}_v (x) \right| dx $$
for all $v > 1$, hence $v \in [0, 1]$ and Lemma \ref{lem12} follows. The analogous property for some $v' \in [0, 1]$ of course also holds in $(\gamma, \beta]$, where $\lim_{x \rightarrow \gamma^+} f^{**} (x) = \tau > 0$ for some $\tau$.

\end{proof}

Quite analogously to the proof of Lemma \ref{lem7} we conclude that for $f^{**}$ we must have $\delta + \tau = 1$.

Next for given $\alpha, \beta, \gamma, \delta$ we determine $v$ and $v'$ such that 
$$ \int_\alpha^\gamma \left| \tilde{f}_v(x) \right| dx \quad  \textnormal{and}  \quad \int_\gamma^\beta \left| \tilde{f}_{v'}(x) \right| dx$$
become minimal. This again is an easy minimisation of a quadratic polynomial and leads to 
\begin{align*}
v &= a - \frac{1}{2} - \frac{\delta}{(\gamma - \alpha) a^{t-1}} \quad \textnormal{and}\\
v' &= a - \frac{1}{2} - \frac{1 - \delta}{(\beta - \gamma) a^{t-1}} \quad .
\end{align*}

Finally we minimise
$$ \int_\alpha^\gamma \left| \tilde{f}_v(x) \right| dx + \int_\gamma^\beta \left| \tilde{f}_{v'}(x) \right| dx $$
with $v$ and $v'$ like above, first with respect to $\delta$ and then with respect to $\gamma$ (again just minimising quadratic polynomials) and obtain $\delta = \frac{1}{2}$ and $\gamma = \frac{\alpha + \beta}{2}$, and hence $v = v' = a - \frac{1}{2} - \frac{a^{1-t}}{\chi}$ (with $\chi = \beta - \alpha$) as optimal parameters. For these choices of parameters we have 
$$ \int_\alpha^\beta \left| f^{**}(x) \right| dx  \ge \int_\alpha^\beta \left| \tilde{f}_{v}(x) \right| dx  = \frac{\chi (4a - a^t \chi)}{16 a} \quad .$$

Like in Lemma \ref{lem10} for $f^*$ we show that all parts $Q'$ of $f^*$ must have the same length, say $\chi$, and all parts $Q''$ of $f^*$ must have the same length, say $\tau$, and therefore
\begin{align*}
\int_0^1 \left| f^{**}(x) \right| dx  \ge & (a-2)  a^{t-1}  \frac{\chi (4a - a^t \chi)}{16 a} +\\
 & + 2 a^{t-1} \frac{\tau^2}{4} (a-2) a^{t-1}
\end{align*}
with $(a-2) a^{t-1} \chi + 2 a^{t-1}  \tau = 1$.

Note that for $\chi$ we further have the condition that $ 0 \le v = a - \frac{1}{2} - \frac{a^{1-t}}{\chi} \le 1$, hence
$$ \frac{a^{1-t}}{a- \frac{1}{2}} \le \chi \le \frac{a^{1-t}}{a - \frac{3}{2}} \quad .$$
$\int_0^1 \left| f^{**}(x) \right| dx $ is a quadratic polynomial in $\chi$ with positive leading coefficient and has a minimum in $\chi = \frac{2 (a-3) a^{1-t}}{7 - 8a + 2a^2}$  which, however, is less than $\frac{a^{1-t}}{a- \frac{1}{2}}$ for $a < 4$ (remember that we restrict to $3 < a < 4$).

Hence a lower bound for $\int_0^1 \left| f^{**}(x) \right| dx$ under all conditions on $\chi$ is obtained for $\chi = \frac{a^{1-t}}{a-\frac{1}{2}}$ (hence $v=0$), and gives
$$ \int_0^1 \left| f^{**}(x) \right| dx \ge \frac{(a-2)(8a + 3)}{8 (1-2a)^2} \quad .$$

This alltogether results in 

\begin{lemma}\label{lem13}
$$ \min_{f \, \textnormal{strongly admissible}} \int_0^1 \left| f (t) \right| dt \ge \frac{(a-2)(8a+3)}{8(1-2a)^2}$$
\end{lemma}

\section{Proof of the Theorem}\label{sec3}

\noindent \textit{Proof of the Theorem.}
Let $x_1, \ldots, x_N$ in $[0,1), N = [a^t]$ for some real $a$ with $3 < a < 4$ and some $t \in \NN$, and $A, A_0, A_1, A_2$ as defined in Section \ref{sec2}. We consider 
$$ P (t) := \int_0^1 \left( \max_{n \in A} D_n (x) - \min_{n \in A} D_n (x) \right) dx .$$

By Lemma \ref{lem1} we obtain
\begin{align*}
P (t)  \ge & \, \frac{1}{2} \int_0^1 \left( \max_{n \in A_2} D_n (x) - \min_{n \in A_2} D_n (x) \right) dx + \\
+ & \,  \frac{1}{2} \int_0^1 \left( \max_{n \in A_0} D_n (x) - \min_{n \in A_0} D_n (x) \right) dx + \\
+ & \, \frac{1}{2} \int_0^1 \left| \max_{n \in A_2} D_n (x) - \max_{n \in A_0} D_n (x) \right| dx + \\
+ & \, \frac{1}{2} \int_0^1 \left| \min_{n \in A_2} D_n (x) - \min_{n \in A_0} D_n (x) \right| dx \quad .
\end{align*}\vspace{0.5em}

From the considerations in Section \ref{sec2}, especially also from Remark \ref{rem1} we know that
$$ f(x) := \max_{n \in A_2} D_n (x) - \max_{n \in A_0} D_n (x) $$
is strongly admissible.

Hence by Lemma \ref{lem13}
$$ \int_0^1 \left| f(x) \right| dx \ge \chi_a$$
with $ \chi_a := \frac{(a-2)(8a+3)}{8(1-2a)^2} $ .

Of course the same arguments used in Section \ref{sec2} can be applied to the function
$$ g(x):= \min_{n \in A_2} D_n (x) - \min_{n \in A_0} D_n (x) , $$
so that we obtain

\begin{align*}
P (t) \ge & \frac{1}{2} \int_0^1  \left( \max_{n \in A_2} D_n (x) - \min_{n \in A_2} D_n (x) \right) dx +\\
+ & \frac{1}{2} \int_0^1 \left( \max_{n \in A_0} D_n (x) - \min_{n \in A_0} D_n (x) \right) dx + \\
+ & \chi_a \quad .
\end{align*}

Now we apply the same procedure for the first two summands, they can be regarded as certain $P (t-1)$, and proceeding in this way we obtain
$$ P(t) \ge t \cdot \chi_a = \log N  \frac{\chi_a}{\log a} \quad .$$

Hence by the definition of $P(t)$ there exist $n \le N$ and $x \in [0,1)$ such that
$$ \left| D_n (x) \right| \ge \log N \frac{\chi_a}{2 \, \log a} \quad , $$
that means, there is an $n \le N$ such that 
$$ n \, D^*_n \ge \log N \frac{\chi_a}{2 \, \log a} \quad .$$
Of course from this we can deduce that for every infinite sequence $x_1, x_2, \ldots$ in $[0,1)$ we have
$$ n \, D^*_n \ge \log n \frac{\chi_a}{2 \, \log a} $$
for infinitely many $n$.

If we choose now $a=3.71866\ldots$ then we obtain $n D_n^* \ge \log n \cdot 0.0646363\ldots $ for infinitely many $n$.

The proof is finished.

\end{document}